\documentclass{imsart}

\RequirePackage[OT1]{fontenc}
\RequirePackage{amsthm,amsmath,amsfonts}
\RequirePackage[numbers]{natbib}
\RequirePackage[colorlinks,citecolor=blue,urlcolor=blue]{hyperref}
\usepackage[dvips]{graphics}

\newtheorem{theorem}{Theorem}[section]

\newtheorem{corollary}{Corollary}[section]
\newtheorem{definition}{Definition}

\newtheorem{conjecture}{Conjecture}[section]

\newtheorem{remark}{Remark}

\newcommand{\mJ}{{\mathcal J}}

\newcommand{\mY}{{\mathcal Y}}

\newcommand{\bR}{{\mathbb R}}

\newcommand{\bE}{{\mathbb E}}

\newcommand{\mC}{{\mathcal C}}

\newcommand{\mR}{{\mathcal R}}

\begin{document}

\begin{frontmatter}
\title{Network iso-elasticity and weighted $\alpha$-fairness}
\runtitle{Iso-elasticity and weighted $\alpha$-fairness}
\maketitle

\begin{aug}
\author{\fnms{Sem} \snm{Borst}\ead[label=e1]{sem@alcatel-lucent.com}},
\author{\fnms{Neil} \snm{Walton}\ead[label=e2]{n.s.walton@uva.nl}}
\and
\author{\fnms{Bert} \snm{Zwart}\ead[label=e3]{Bert.Zwart@cwi.nl}}

\runauthor{Borst et al.}

\affiliation{Bell Laboratories, University of Amsterdam and CWI.}

\address{Bell Laboratories, Alcatel-Lucent, 600 Mountain Avenue, Murray Hill, NJ 07974, USA\\
\printead{e1}\\}

\address{University of Amsterdam, Science Park 904, 1098 XH Amsterdam\\
\printead{e2}}

\address{CWI, Science Park 123, 1098 XG Amsterdam\\
\printead{e3}}
\end{aug}

\begin{abstract}
When a communication network's capacity increases, it is natural to want the bandwidth allocated to increase to exploit this capacity. But, if the same relative capacity increase occurs at each network resource, it is also natural to want each user to see the same relative benefit, so the bandwidth allocated to each route should remain proportional. We will be interested in bandwidth allocations which scale in this \textit{iso-elastic} manner and, also, maximize a utility function. 

Utility optimizing bandwidth allocations have been frequently studied, and a popular choice of utility function are the weighted $\alpha$-fair utility functions introduced by Mo and Walrand \cite{MoWa00}. Because weighted $\alpha$-fair utility functions possess this iso-elastic property, they are frequently used to form fluid models of bandwidth sharing networks. In this paper, we present results that show, in many settings, the only utility functions which are iso-elastic are weighted $\alpha$-fair utility functions. 

Thus, if bandwidth is allocated according to a network utility function which scales with relative network changes then that utility function must be a weighted $\alpha$-fair utility function, and hence, a control protocol that is robust to the future relative changes in network capacity and usage ought to allocate bandwidth inorder to maximize a weighted $\alpha$-fair utility function.

%
%
\end{abstract}
\end{frontmatter}
\section{Introduction}
It is often claimed that both Internet traffic and capacity undergo multiplicative year-on-year increase. Such a rate may not be sustainable over time, even so, it is a pervasive point that the Internet's demand and service rate changes multiplicatively. Given this, in what way should a communication networks resources be allocated inorder to scale with such multiplicative effects?

Mo and Walrand \cite{MoWa00} introduced a parametrized family of utility functions called the \textit{weighted $\alpha$-fair utility functions}.
A utility function orders the preferences of different network states: A network state is better than a second network state if it has higher utility.
When maximized over a network's capacity set, a utility function provides a solution which can be used to allocate network resources.
Congestion control protocols such as TCP are argued to allocate bandwidth in this way \cite{Lo00}, and congestion protocol models have been designed that provably reach such an operating point \cite{KMT98}.

Suppose that a performance improvement is made whereby the network capacity is doubled. Although capacity increases, the criteria by which capacity is evaluated and shared should remain the proportionate. We note that, in particular, the weighted $\alpha$-fair utility functions obey this scaling property. But are there more? 
Such a utility function should, also, scale well when the traffic increases. That is, multiplicatively increasing the number of network flows should not alter the preferred allocation of network resources. This is a desirable property: regardless of relative changes, the proportion of network resource each route receives remains the same. Given that the Internet traffic is increasing multiplicatively, it is desirable that each flow has a share of resource that remains proportionate. 
Once again, the weighted $\alpha$-fair utility functions will scale in this way. So are the weighted $\alpha$-fair utility functions the only utility functions that satisfy these scalability properties? We mathematically formulate scaling properties and general settings where this is provably true. We thus provide a theoretical basis to the claim that, if bandwidth is allocated according to a network utility function which scales with relative network changes, then that utility function must be a weighted $\alpha$-fair utility function, and hence, a control protocol that is robust to the future relative changes in network capacity and usage must to allocate bandwidth inorder to maximize a weighted $\alpha$-fair utility function.

Economists Arrow \cite{Ar65} and Pratt \cite{Pr64} formulated the utility functions that satisfy a certain \textit{iso-elastic property}. In words, an individual is iso-elastic if his preferences enter a bet are unaltered by a multiplicative changes in the bet's stakes and rewards. 
Arrow and Pratt parametrize the set of iso-elastic utility functions on $\bR_+$. It is an immediate check that the iso-elastic utility functions are precisely the summands of a weighted $\alpha$-fair utility function. 
Thus the reason a weighted $\alpha$-fair utility function has good scaling properties is because of the iso-elasticity property inherent in its summands. But conversely, if a utility function allocates network resources proportionately then must it be weighted $\alpha$-fair? The main results of this paper provide settings where scaling properties are equivalent to a utility function being weight $\alpha$-fair.

The work of Mo and Walrand has received a great deal of attention. This is chiefly because a number of known network fairness criteria correspond to maximizing a weighted $\alpha$-fair utility function: known Internet equilibria, \textit{TCP fairness} \cite{Lo00}, for $\alpha=2$; the work Kelly \cite{Ke97}, weighted proportional fairness, for $\alpha=1$; the work of Bertsemas and Gallager \cite{BeGa87}, max-min fairness, for $w=1$ with $\alpha\rightarrow\infty$ and maximum throughput for $w=1$ with $\alpha\rightarrow 0$. But, perhaps, the explicit reason for the attention on weighted $\alpha$-fairness has not been the above iso-elastic property. The iso-elastic property is frequently exploited inorder to studying the limit and stability behaviour of associated stochastic network models \cite{BoMa01,GW09,MaRo99,KKLW07i,EBZ07,YeYa10}; some authors have begun to explicitly cite the iso-elastic property \cite{GLS07, LM09}, and other authors have attempted to derive weighted $\alpha$-fairness and other fairness criteria from axiomatic properties \cite{LKCS10,UcKu09}. Even so, there appears to be little discussion of this iso-elastic scaling property, its equivalence with weighted $\alpha$-fairness, and what this implies for the performance achieved by a communication network and associated stochastic models. Addressing this point is the principle contribution of this paper.

The results and sections of this paper are organized as follows. In Section \ref{sec:net}, we define the topology and capacity constraints of a communication network. In Section \ref{sec:NetUtil}, we define what we mean by a network utility function, network utility maximization and we define the weighted $\alpha$-fair utility functions. In Section \ref{sec:iso}, we formally define the network iso-elasticity property and the flow scalability property. In Section \ref{sec:isothrm}, we prove that a network utility function is network iso-elastic iff it is weighted $\alpha$-fair. In Section \ref{sec:Scale}, we consider a network topology where each link has some dedicate local traffic flow. On this topology, we show that a network utility maximizing allocation satisfies flow scalability iff it is a weighted $\alpha$-fair maximizer. In Section \ref{sec:access}, in addition to the flow scalability property, we define an access scalability property. We prove the only utility functions that are flow scalable and access scalable are the weighted $\alpha$-fair utility functions. In proving these results we lay claim to the statement that if bandwidth is allocated according to a network utility function which scales with relative network changes then that utility function must be a weighted $\alpha$-fair utility function. Thus a control protocol that is robust to the future relative changes in network capacity and usage must allocate bandwidth to maximize a weighted $\alpha$-fair utility function.

\section{Network Structure}\label{sec:net}
We define the topology and capacity constraints of a network. We suppose that there are a set of links indexed by $\mJ$. The positive vector $C=(C_j:j\in\mJ)$ gives the capacity of each link in the network. A route consists of a set of links. We let $\mR$ denote the set of routes. The topology of the network is defined through a $|\mJ|\times |\mR|$ matrix $A$. We let $A_{jr}=1$ if route $r$ uses link $j$ and we let $A_{jr}=0$ otherwise. We let the positive vector $\Lambda=(\Lambda_r:r\in\mR)$ denote the bandwidth allocated to each route. At link $j$ these must satisfy the capacity constraint 
\begin{equation*}
 \sum_{j\in\mJ} A_{jr} \Lambda_r \leq C_j,
\end{equation*}
and thus the set feasible bandwidth allocations is given by
\begin{equation*}
 \mC = \{ \Lambda \in \bR_+^{\mR}: A\Lambda \leq C \}.
\end{equation*}

\section{Network Utility}\label{sec:NetUtil} In communication network, we could allocate resources to achieve the highest aggregate throughput. Although a high data rate is achieved, some networks flows maybe be starved. Over the last decade, there has been interest in allocating network resources in a \textit{fair} way. Essentially, fairness is achieved by allocating a positive share of the networks resources to each flow.
To achieve this it was recommended to allocate resources inorder to maximize utility, see Kelly \cite{Ke97}. In such a network, each flow expresses its demand via a strictly increasing concave utility function. A utility function will have a higher slope for smaller throughputs. When maximized, this function can be used express a higher demand for the networks resources. Thus, from this, a form of fairness is achieved. 

The average utility of a network flow is
\begin{equation}
U(y,\Lambda)=\frac{1}{\sum_{i\in\mR} y_i}\sum_{r\in\mR} y_r U_r\big(\frac{\Lambda_r}{y_r}\big).\footnote{In this paper, we chose to express the utility of the average network user rather than the aggregate utility of the network. We note the optimum of these functions will be unaffected when maximizing over $\Lambda$.}\label{agutil2}
\end{equation}
Here $\mR$ indexes the routes of a network; $y_r$ gives the number of flows present on route $r$; $\Lambda_r$ gives the flow rate allocated to route $r$, which is then shared equally amongst the flows present on the route, and finally, $U_r$ gives the utility function of each route $r$ user. We call a utility function of the form \eqref{agutil2}, a \textit{network utility function}.

We assume throughout this paper that each utility function $U_r$ is increasing, once differentiable and strictly concave.

To the average user, a network utility function orders preferences of different network states, i.e., $(y,\Lambda)$ is as least as good as $(\tilde{y},\tilde{\Lambda})$ if $U(y,\Lambda) \geq U(\tilde{y},\tilde{\Lambda})$.
Given the flows present in a network, the best network state is one that maximizes utility. That is, $\Lambda(y)$ the optimum of the \textit{network utility maximization}.
\begin{equation}
 \max\quad U(y,\Lambda)\quad \text{over} \quad \Lambda \in \mC.\label{max bandwidth}
\end{equation}
Here $\mC\subset \bR_+^\mR$ is the set of feasible bandwidth allocations. In general, set of bandwith allocations will depend on the topology and capacity constraints of a communication network.

The \textit{weighted $\alpha$-fair} utility functions introduced by Mo and Walrand \cite{MoWa00} are given by
\begin{align}
W_{\alpha,w}(y,\Lambda):=
\begin{cases}
\frac{1}{\sum_{i\in\mR} y_i}\sum_{r\in\mR} \frac{w_r y_r}{1-\alpha} \big(\frac{\Lambda_r}{y_r}\big)^{1-\alpha} & \text{if } \alpha\neq 1, \\
\frac{1}{\sum_{i\in\mR} y_i}\sum_{r\in\mR} w_r y_r \log \big(\frac{\Lambda_r}{y_r}\big)& \text{if } \alpha= 1,
\end{cases}
\end{align}
for $\Lambda=(\Lambda_r:r\in\mR), y=(y_r:r\in\mR)\in\bR_+^\mR$ and parametrized by $\alpha>0$ and $w\in\bR_+^\mR$.
As mentioned, the weighted $\alpha$-fair class has proved popular as it provides a spectrum of fairness criteria which contains proportional fairness ($\alpha=w_i=1$), TCP fairness ($\alpha=2$, $w_i=\frac{1}{T^2_i}$), and converges to a maximum throughput solution ($\alpha\rightarrow 0$, $w_i=1$) and max-min fairness ($\alpha\rightarrow\infty$, $w_i=1$).




\section{Network Iso-elasticity and flow scalability} \label{sec:iso}
We now define the two main scaling properties which we will use: network iso-elasticity and flow scalability. Both encapsulate a simple idea, if we multiplicatively change the available network capacity then we will allocate resources proportionately.

A network utility function ranks the set of network states through the ordering induced by $U(y,\Lambda) \geq U(\tilde{y},\tilde{\Lambda})$. We say a network utility function is network iso-elastic if this ordering is unchanged by a multiplicative increase in the available bandwidth. More formally, 
\begin{definition}
We define a utility function to be \textit{network iso-elastic} if, $\forall a>0$,
\begin{align}\label{Util 1}
{U}(y,\Lambda)\geq {U}(\tilde{y},\tilde{\Lambda})\quad\text{iff}\quad {U}(y,a\Lambda)\geq {U}(\tilde{y},a\tilde{\Lambda}), 
\end{align}
for each $\Lambda,\tilde{\Lambda},y,\tilde{y}\in\bR_+^\mR$.
\end{definition}
We note that this is equivalent to the expression, where we multiplicatively scale the number of flows, $\forall a>0$,
\begin{align}\label{net iso}
{U}(y,\Lambda)\leq {U}(\tilde{y},\tilde{\Lambda})\quad\text{iff}\quad {U}(ay,\Lambda)\leq {U}(a\tilde{y},\tilde{\Lambda}),
\end{align}
for each $\Lambda,\tilde{\Lambda},y,\tilde{y}\in\bR_+^\mR$.

Network iso-elasticity requires that the utility of each network state $(y,\Lambda)$ scales. What if we only wish the optimal allocation to scale? 
\begin{definition}
We say a utility function optimized on capacity set $\mC$ is \textit{flow scalable} if $\Lambda(y)$ the solution to the optimization problem \eqref{max bandwidth} is such that
\begin{equation}\label{flow scale 1}
 \Lambda(y)=\Lambda(ay),\qquad \forall a>0.
\end{equation}
\end{definition}
That is the bandwidth allocated to each route is unchanged by multiplicative changes in the number of flows on each route. Making explicit that $\Lambda(y;\mC)$ is the solution to \eqref{max bandwidth} when optimizing over $\mC$ and defining $a\mC=\{a\Lambda: \Lambda\in \mC \}$.  We note that network scalability condition is equivalent to the condition that allocated bandwidth scales proportionately with capacity increases, i.e.,
\begin{equation}\label{flow scale 2}
 a\Lambda(y;\mC)=\Lambda(y;a\mC),\qquad \forall a>0.
\end{equation}

\section{Network iso-elasticity and weighted $\alpha$-fairness}\label{sec:isothrm}
In this section, we prove the first result of this paper. We prove a network utility function satisfying this iso-elastic property is, up to an additive constant, a weighted $\alpha$-fair utility function.

\begin{theorem} \label{iso thrm}
For $U(y,\Lambda)$ the network utility function \eqref{agutil2} the following are equivalent, \vspace{0.2cm}\\
\noindent i) $U(y,\Lambda)$ is network iso-elastic.\vspace{0.1cm}\\
\noindent ii) up to an additive constant, $U(y,\Lambda)$ is weighted $\alpha$-fair, i.e., there exist $\alpha>0$, $w\in\bR_+^\mR$ and constant $c\in\bR$ such that 
\begin{equation*}
 U(y,\Lambda)=W_{\alpha,w}(y,\Lambda)+c.
\end{equation*}

\end{theorem}
\begin{proof}
It is an immediate calculation that ii) implies i). We now prove conversely that i) implies ii). The key idea is to prove that the map $\phi_a:U(y,\Lambda)\mapsto U(y,a\Lambda)$ is linear.

For simplicity, we suppose $\sum_r y_r =\sum_r \tilde{y}_r =1$; we, also, let $x_r=\frac{\Lambda_r}{y_r}$ and $\tilde{x}_r=\frac{\tilde{\Lambda}_r}{\tilde{y}_r}$ for $r\in\mR$, and we let $I$ and $\tilde{I}$ be random variables with distribution $y$ and $\tilde{y}$, respectively. One can see the value of the utility function
\begin{equation*}
\bE U_I(x_I)=\sum_{r\in\mR} y_r U_r(x_r),
\end{equation*}
induces an ordering on the set of pairs $(y,\Lambda)$.
 Network iso-elasticity \eqref{Util 1} states that, $\forall a>0$, $\bE U_I(x_I)$ and $\bE U_I(ax_I)$ induce the same ordering on elements $(I,x)$. 
Thus \eqref{Util 1} implies, for each $a>0$, there exists an increasing function $\phi_a$ such that
\begin{equation*}
 \phi_a(\bE U_I(x_I))=\bE U_I(ax_I),\qquad \forall (I,x).
\end{equation*}

Now let us show that $\phi_a(\cdot)$ is linear. Let $\hat{I}$ be the random variable such that
\begin{equation*}
\hat{I}=\begin{cases} I &\text{ with probability } p,\\ \tilde{I}& \text{ with probability }(1-p).\end{cases}
\end{equation*}
Let $\tau=\bE U_I(x_I)$ and $\sigma=\bE U_{\tilde{I}}(x_{\tilde{I}})$. Note $p \tau + (1-p) \sigma = \bE U_{\hat{I}}(x_{\hat{I}})$. Now observe
\begin{align*}
 \phi_a(p \tau\! +\! (1-p)\sigma) = \bE U_{\hat{I}}(ax_{\hat{I}}) &= p\bE U_I(ax_I)\! + \!(1-p) \bE U_{\tilde{I}}(x_{\tilde{I}})\\
&=p \phi_a(\tau)\! +\!(1-p) \phi_a(\sigma).
\end{align*}
Thus $\phi_a(\cdot)$ is an increasing linear function and so, for $a>0$,
\begin{equation}
\bE U_I(ax_I)=b_a\bE U_I(x_I) + d_a,\qquad \forall (I,x),\label{useful statement}
\end{equation}
for some $b_a>0$ and $d_a\in\bR$. 
Statement \eqref{useful statement} applies to $I\equiv r$, for $r\in\mR$. Thus
\begin{equation*}
U_r(ax_r)=b_a U_r(x_r) + d_a,\qquad \forall x_r\in\bR_+.
\end{equation*}
 We wish to remove the seemingly arbitrary functions $b_a$ and $d_a$. Differentiating with respect to $x_r$ gives,
\begin{equation}
U'_r(ax_r)=b_a U'_r(x_r),\qquad \forall x_r\in\bR_+.\label{diffy1}
\end{equation}
As $U'_r$ is a decreasing function it is differentiable at at least one point \cite[Theorem 7.2.7]{Du02}. Differentiating again at this point gives
\begin{equation*}
U''_r(ax_r)=b_a U''_r(x_r),\qquad \forall a>0.
\end{equation*}
Thus, as $a$ is arbitrary, $U''_r(x_r)$ exists for all $x_r>0$. Dividing both terms leads to the equation
\begin{equation*}
a\frac{U''_r(ax_r)}{U'(ax_r)}= \frac{U''_r(x_r)}{U'(x_r)},\qquad \text{for } x_r>0,\; a>0.
\end{equation*}
Taking $x_r=1$ and defining $\alpha:=-\frac{U''_r(1)}{U'_r(1)}$, we have the differential equation
\begin{equation*}
 a\frac{U''_r(a)}{U'_r(a)}=-\alpha_r,\qquad \text{for } a>0.
\end{equation*}
Dividing both sides by $a$ and integrating with respect to $a$ gives
\begin{equation*}
 \log U'_r(a) = -\alpha \log a + \log w_r\quad\implies\quad U'_r(a)= w_r a^{-\alpha_r}
\end{equation*}
Here $\log w_r$ is an appropriate additive constant. Integrating one more time gives and letting $a=x_r$ gives
\begin{equation}
 U_r(x_r)=
\begin{cases}
w_r \frac{x_r^{1-\alpha_r}}{1-\alpha_r} + c_r &\text{if } \alpha_r \neq 1,\\
w_r \log x_r + c_r &\text{if } \alpha_r=1.
\end{cases}\label{diffy2}
\end{equation}

It remains to show $\alpha_r=\alpha$ for all $r\in\mR$. Observe that \eqref{Util 1} includes the statement $U_r(x_r)\geq U_i(x_i)$ iff $U_r(ax_r)\geq U_i(ax_i)$, which can only hold if $\alpha_i=\alpha_r$. So $\alpha_r\equiv \alpha$ for some $\alpha>0$ and thus for all $y\in\bR^\mR_+$ and $\Lambda\in\bR^\mR_+$
\begin{align*}
{U}(y,\Lambda)&={U}\Big(\frac{y}{\sum_i y_i},\Lambda\sum_i y_i\Big)\\
&=
\begin{cases}
\frac{1}{\sum_i y_i}\sum_{r\in\mR} \frac{w_r y_r}{1-\alpha} \big(\frac{\Lambda_r}{y_r}\big)^{1-\alpha} & \text{ if } \alpha\neq 1, \\
\frac{1}{\sum_i y_i}\sum_{r\in\mR} w_r y_r \log \big(\frac{\Lambda_r}{y_r}\big)& \text{ if } \alpha= 1,
\end{cases}
\end{align*}
as required.
\end{proof}
\begin{remark}
We note that an almost identical result and proof holds for the aggregate utility function
\begin{equation*}
 U(y,\Lambda)=\sum_{r\in\mR} y_r U_r\Big(\frac{\Lambda_r}{y_r} \Big).
\end{equation*}
The only exception would be that the above result does not hold for the weighted proportionally fair case (where $\alpha=1$).
\end{remark}

\section{Scaleability and weighted $\alpha$-fairness in Networks with Local Traffic}\label{sec:Scale}
In this section, we demonstrate that flow scalability is equivalent to optimizing a weighted $\alpha$-fair utility function for two specific network topologies: a linear network and a network satisfying the local traffic condition. 

A \textit{linear network} consist of a routes $r_0,...,r_K$ and links $j_1,...,j_K$ for $K>1$. Route $r_0$ uses all the links, $A_{jr_0}=1\; \forall j\in\mJ$, whilst for $k=1,...,K$ each route $r_k$ only uses link $j_k$, $A_{j_kr_k}=1$ and $A_{jr_k}=0$ otherwise. The \textit{local traffic condition} insists that each link has a route that uses that link only, $\forall j$ $\exists r$ such that $A_{jr}=1$ and $A_{lr}=0$ $\forall l\neq j$. We will also assume a network that satisfies the local traffic condition has two or more links and is connected, that is for all links $j,l\in\mJ$ there exists a sequence of links and routes $j_0,r_0,j_1,r_1,...,r_{K-1},j_{K}$ such that $j_0=j$ and $j_K=l$, and $A_{j_k,r_k}=A_{j_{k+1},r_k}=1$ for $k=0,..,K-1$.

We now prove that flow scalability is equivalent to weighted $\alpha$-fairness for linear networks.

\begin{theorem}\label{linear theorem}
Assuming, for $r\in\mR$, $U'_r(x)$ has range $(0,\infty)$, a linear network is flow scalable iff it maximizes a weighted $\alpha$-fair utility function.
\end{theorem}
\begin{proof}
It is immediate that the weighted $\alpha$-fair utility is flow scalable. Let us show the converse. 

The bandwidth a linear network allocates is the solution to the optimization
\begin{align}
 &\max &&\sum_{r\in\mR} y_r U_r\Big(\frac{\Lambda_r}{y_r}\Big) \label{linear fn}\\
 &\text{subject to}&& \Lambda_{r_0}+\Lambda_{r_k} \leq C_{r_k}, \quad k=1,...,K \label{linear fn2}.
\end{align}
Here, we let $C_{r_k}$ denote the capacity of the link on the intersection of route $r_k$ and route $r_0$. Assuming $y_r>0$ $\forall r\in\mR$, the solution to this optimization satisfies
\begin{align}
 C_{r_k}&=\Lambda_{r_0}+\Lambda_{r_k}, \qquad k=1,...,K\\
 \text{and}\quad  U_{r_0}'\Big(\frac{\Lambda_{r_0}}{y_{r_0}}\Big)&=\sum_{k=1}^K U_{r_k}'\Big(\frac{\Lambda_{r_k}}{y_{r_k}}\Big). \label{linear diff1}
\end{align}
If the solution is flow scalable then $\Lambda$ must also satisfy
\begin{equation}
 U_{r_0}'\Big(\frac{\Lambda_{r_0}}{ay_{r_0}}\Big)=\sum_{k=1}^K U_{r_k}'\Big(\frac{\Lambda_{r_k}}{ay_{r_k}}\Big),\qquad \forall a>0.\label{linear diff2}
\end{equation}

We observe that the solutions for a linear network allow a suitably large array of values for $U'_r\big( \frac{\Lambda_r}{y_r}\big)$ and $\Lambda_r$. In particular, for any $p_1,...,p_K\in(0,\infty)$ and $\Lambda_{r_0},...,\Lambda_{r_K}\in(0,\infty)$, we can choose a $y_r>0$ such that $p_k=U'_{r_k}\big(\frac{\Lambda_{r_k}}{y_{r_k}}\big)$, $k=1,...,K$. Letting $p_0=\sum_{k=1}^K p_k$, we can also chose $y_{r_0}$ such that $p_0=U'_{r_0}\big(\frac{\Lambda_{r_0}}{y_{r_0}} \big)$. Thus, this choice of $\Lambda$ gives the unique solution to \eqref{linear fn} for the $y$ given.

The remainder of the proof is similar to that of Theorem \ref{iso thrm}. Once again, the key idea is to prove the map $\phi_a^r:U_r'(\frac{\Lambda_r}{y_r})\mapsto U_r'(\frac{\Lambda_r}{ay_r})$ is linear. As $U'_r(x)$ must be continuous and decreasing $\phi_a^r(\cdot)$ is an increasing continuous function. Also, note that $\phi_a^(0)=0$ because $U'_r(\infty)=0$.

With $p_r=U_r\Big(\frac{\Lambda_r}{y_r}\Big)$ as above, \eqref{linear diff1} and \eqref{linear diff1} state that
\begin{equation}\label{phi lin 1}
 \phi_a^{r_0}(p_0)=\sum_{k=1}^K \phi_a^{r_k}(p_{k})
\quad
\text{for}
\quad
 p_0=\sum_{k=1}^K p_{k}.
\end{equation}
Taking $p_0=p$, $p_k=p$ and $p_{k'}=0$ for $k'\neq k$, we see that
\begin{equation}\label{phi lin 2}
 \phi_a^{r_0}(p)=\phi_a^{r_k}(p),\qquad \forall p\geq 0.
\end{equation}
Taking $p_0=p$, $p_1=\hat{p}$ and $p_2=p-\hat{p}$ for $p>\hat{p}\geq 0$, \eqref{phi lin 1} and \eqref{phi lin 2} together imply
\begin{equation} \label{phi lin 3}
 \frac{\phi_a^{r}(p)-\phi_a^{r}(\hat{p})}{p-\hat{p}}= \frac{\phi_a^{r_0}(p-\hat{p})}{p-\hat{p}}.
\end{equation}
As $\phi_a^{r_0}$ is increasing it is differentiable at some point (see \cite[Theorem 7.2.7]{Du02}), then by \eqref{phi lin 3} it is differentiable at all points and moreover its derivative is constant. In otherwords, $\phi^r_a$ must be linear,
\begin{equation*}
 \phi_a^r(p_r)=b_a p_r
\end{equation*}
for some function $b_a$. Thus
\begin{equation*}
U'_r(ax_r)=b_a U'_r(x_r),\qquad \forall x_r\in\bR_+.
\end{equation*}

By same argument used to derive \eqref{diffy2} from \eqref{diffy1} in Theorem \ref{iso thrm}, we have that
\begin{equation*}
 U_r(x_r)=
\begin{cases}
w_r \frac{x_r^{1-\alpha_r}}{1-\alpha_r} + c_r &\text{if } \alpha_r \neq 1,\\
w_r \log x_r + c_r &\text{if } \alpha_r=1.
\end{cases}
\end{equation*}
$\alpha_r>0$, $w_r>0$ and $c_r\in\mR$. 


Finally note, if $\alpha_{r_k}\neq\alpha_{r_0}$ for some $r_k$ then the equalities \eqref{linear diff2} and \eqref{linear diff1} cannot hold, for all $a>0$. Thus, we now see the optimization (\ref{linear fn}-\ref{linear fn2}) must have been a weighted $\alpha$-fair optimization.
\end{proof}


From this argument for linear networks, we extend our result to any network satisfying the local traffic condition.

\begin{corollary}
Assuming, for $r\in\mR$, $U'_r(x)$ has range $(0,\infty)$, a network satisfying is the local traffic condition is flow scalable iff it maximizes a weighted $\alpha$-fair utility function.
\end{corollary}
\begin{proof}
It is immediate that the weighted $\alpha$-fair utility is flow scalable. To show the converse, we show a local traffic network can be reduced to a linear network. Take any route $r_0\in\mR$ which uses two or more links. Let them be $j_1,...,j_K$  and let $r_1,...,r_K$ be local traffic routes for each respective link. Set $y_r=0$ for all $r\notin\{r_0,...,r_K\}$. The resulting network is a linear network, and thus, by Theorem \ref{linear theorem}, the utility function associated with each route of this subnetwork is of the form
\begin{equation*}
 U_r(x_r)=
\begin{cases}
w_r \frac{x_r^{1-\alpha_r}}{1-\alpha_r} + c_r &\text{if } \alpha_r \neq 1,\\
w_r \log x_r + c_r &\text{if } \alpha_r=1.
\end{cases}
\end{equation*}
$\alpha_r>0$, $w_r>0$ and $c_r\in\mR$. Here $\alpha_r$ is same for all routes in this linear subnetwork, i.e. $\alpha_{r_k}=\alpha_{r_0}$ for all $k=1,...,K$ . 

We show, that since our network is connected, $\alpha_r$ does not depend on the route chosen. We know $\alpha_r=\alpha_{r(j)}$ for each local traffic route $r(j)$ intersecting route $r$. Take two routes $r$ and $\tilde{r}$.  Since our network is connect, for $j$ a link on route $r$ and $l$ a link on route $\tilde{r}$, there are a sequence $j_0,r_0,j_1,r_1,...,r_{K-1},j_{K}$ connecting $j_0=j$ and $j_K=l$. For each $k=0,...,k-1$, the subnetwork formed by $r_k$, $j_k$ and $j_{k+1}$ and their respective local traffic routes forms a linear network. Thus $\alpha_r$ is constant across all routes and associated local traffic routes in the sequences $j_0,r_0,j_1,r_1,...,r_{K-1},j_{K}$. This, thus, implies $\alpha_r=\alpha_{\tilde{r}}$.
\end{proof}

By simply adding a local traffic route to each link, any network topology can be made to satisfy the local traffic condition. Thus given the generality of this set of networks, it is natural to conjecture that flow scalability is equivalent to weighted $\alpha$-fairness for any network topology.

\begin{conjecture}
A network is flow scalable iff it maximizes a weighted $\alpha$-fair utility function.
\end{conjecture}

\section{Access scalability and weighted $\alpha$-fairness}\label{sec:access}
The network utility maximization problem is
\begin{equation}
 \max\quad U(y,\Lambda)\quad \text{over} \quad \Lambda \in \mC.\label{max bandwidth2}
\end{equation}
Given the flows on each route $y$, this finds an optimal way to allocate bandwidth $\Lambda$. Given the bandwidth to be allocated is $\Lambda$, what is the optimal number of flows permissible on each route? This leads to the maximization
\begin{equation}
 \max\quad U(y,\Lambda)\quad \text{over} \quad y \in \mathcal{Y}.\label{max bandwidth3}
\end{equation}
In this second optimization, we optimize the number of flows accessing routes inorder to guarantee the maximum average utility per flow.

Flow scalability (\ref{flow scale 1}-\ref{flow scale 2}) says, if we multiply the capacity of the network then the rate allocated to each route will multiply by the same amount. Similarly, we will say a utility function is access scalable, if multiplying the bandwidth available to each route means that the number flows accepted on each route will multiply by the same amount. More formally,
\begin{definition}
We say $y(\Lambda)$, the solution to optimization problem \eqref{max bandwidth3}, is access scalable if, for all $a>0$,
\begin{equation*}
 y(a\Lambda)=ay(\Lambda).
\end{equation*}
\end{definition}

In the following theorem, we observe that the only utility function that is both flow scalable and access scalable is a weighted $\alpha$-fair utility function.

\begin{theorem}
A network utility function $U(y,\Lambda)$ is both flow scalable for all $\mC$ and access scalable for all $\mathcal{Y}$ iff it is weighted $\alpha$-fair.
\end{theorem}
\begin{proof}
First, it is an immediate calculation that any weighted $\alpha$-fair utility function is user scalable and access scalability for any set $\mC$ and $\mY$.

Conversely, suppose user scalability and access scalability hold for each $\mC$ and $\mY$. Taking $\mC=\{\Lambda, \tilde{\Lambda}\}$ and $\mY=\{y,\tilde{y}\}$, by user scalability and access scalability, respectively, we have that
\begin{align}
& {U}(y,\Lambda) \leq {U}(y,\tilde{\Lambda}) \quad\text{iff}\quad {U}(ay,\Lambda) \leq {U}(ay,\tilde{\Lambda}),\label{thrm2:1}\\
& {U}(y,\Lambda) \leq {U}(\tilde{y},{\Lambda}) \quad\text{iff}\quad {U}(ay,\Lambda) \leq {U}(a\tilde{y}.{\Lambda})\label{thrm2:2}
\end{align}
By continuity, there exists an $\tilde{a}$ such that
\begin{equation}
 {U}(\tilde{y},\tilde{\Lambda})={U}(\tilde{a}\tilde{y},{\Lambda}).\label{thrm2:3}
\end{equation}
Also, by \eqref{thrm2:1},
\begin{equation}
{U}(\tilde{y},\tilde{\Lambda})={U}(\tilde{a}\tilde{y},\Lambda) \quad\text{iff}\quad {U}(a\tilde{a}y,\Lambda)={U}(ay,\tilde{\Lambda}).\label{thrm2:4}
\end{equation}
Now, (\ref{thrm2:2}-\ref{thrm2:4}) imply the following set of equivalences.
\begin{align*}
 & {U}(y,\Lambda) \leq {U}(\tilde{y},\tilde{\Lambda})={U}(\tilde{a}\tilde{y},{\Lambda}) &\text{[by \eqref{thrm2:3}]}\\
\text{iff}\qquad & {U}(ay,\Lambda)\leq  {U}(a\tilde{a}\tilde{y},\Lambda)  &\text{[by \eqref{thrm2:2}]}\\
\text{iff}\qquad & {U}(ay,\Lambda)\leq  {U}(a\tilde{y},\tilde{\Lambda}).  &\text{[by \eqref{thrm2:4}]}
\end{align*}
Thus, ${U}(y,\Lambda)$ is network iso-elastic and so, by Theorem \ref{iso thrm}, ${U}(y,\Lambda)$ is a weighted $\alpha$-fair utility function.
\end{proof}

%
%

\end{document}